\documentclass{amsart}
\usepackage[margin=1.25in]{geometry}
\usepackage{amsfonts,amssymb,amsmath,amstext,amscd,epsfig}
\usepackage{amsthm,cases,color,graphicx,graphics}
\usepackage{stackengine}
\usepackage{mathtools}
\usepackage{mathrsfs}
\usepackage{dsfont}
\usepackage[usenames,dvipsnames,svgnames,table]{xcolor}
\usepackage[linktocpage=true,colorlinks=true,linkcolor=RubineRed!85!black,citecolor=ForestGreen,urlcolor=green,pdfborder={0 0 0}]{hyperref}

\theoremstyle{plain}
\newtheorem{theorem}{Theorem}[section]
\newtheorem{proposition}[theorem]{Proposition}
\newtheorem{lemma}[theorem]{Lemma}
\newtheorem{remark}[theorem]{Remark}
\numberwithin{equation}{section}

\newcommand*{\dif}{\mathop{}\!\mathrm{d}}
\newcommand{\na}{\nabla}
\newcommand{\la}{\langle}
\newcommand{\ra}{\rangle}
\def \fein {f_{\rm in}}
\def \psin {\psi_{\rm in}}
\def \rin {\rho_{\rm in}}

\def \p {\partial}
\def \R {\mathbb{R}}
\def \LL {\mathcal{L}}
\def \RR {\mathcal{R}}
\def \DD {\mathcal{D}}
\def \ep {\varepsilon}
\def \O {\Omega}

\title[Asymptotics of dissipative kinetic equations in bounded domains]{A note on asymptotics of linear dissipative kinetic equations in bounded domains} 
\author{Yuzhe Zhu} 
\date{October 10, 2024} 
\address{Department of Pure Mathematics and Mathematical Statistics, University of Cambridge, Wilberforce Road, Cambridge CB3 0WA, UK}
\curraddr{Department of Mathematics, University of Chicago, 5734 S University Ave, Chicago Illinois 60637, USA}
\email{yuzhezhu@uchicago.edu}

\begin{document}

%\subjclass[2020]{35B40, 35Q62, 82C31, 82C40, 82C70} asymptotic behavior; PDEs in connection with statistics; Stochastic methods (Fokker-Planck, Langevin, etc.) applied to problems in time-dependent statistical mechanics; Kinetic theory of gases in time-dependent statistical mechanics; Transport processes in time-dependent statistical mechanics

\begin{abstract}
We establish $L^2$-exponential decay properties for linear dissipative kinetic equations, including the time-relaxation and Fokker-Planck models, in bounded spatial domains with general boundary conditions that may not conserve mass. Their diffusion asymptotics in $L^2$ is also derived under general Maxwell boundary conditions. The proofs are simply based on energy estimates together with previous ideas from $L^2$-hypocoercivity and relative entropy methods. 
\end{abstract}
\maketitle

\section{Introduction}\label{intro}
Consider the time-relaxation operator $\LL_1$ and the Fokker-Planck operator $\LL_2$ defined by 
\begin{align*}
&\LL_1f(v) := \int_{\R^d}f(v)\dif\mu-f(v),\\
&\LL_2f(v) := (\na_v-v)\cdot\na_v f(v),
\end{align*} 
for $f(v)\in C_c^\infty(\R^d)$. 
Here we write 
\begin{align*}
\mu(v)&:=(2\pi)^{-d/2}e^{-|v|^2/2},\\
\dif\mu&:=\mu(v)\dif v
\end{align*}
to denote the Gaussian function and the Gaussian measure, respectively. Let $\O$ be a bounded $C^{1,1}$-domain in $\R^d$, and $n_x$ be the unit outward normal vector at $x\in\p\O$. We split the phase boundary $\Sigma:=\p\O\times\R^d$ into the outgoing part $\Sigma_+$, incoming part $\Sigma_-$ and grazing part $\Sigma_0$, which are precisely defined as follows,  
\begin{align*}
\Sigma_\pm&:= \left\{ (x,v)\in\Sigma :\, \pm\;\! n_x\cdot v >0 \right\}, \\
\Sigma_0&:= \left\{ (x,v)\in\Sigma :\, n_x\cdot v =0 \right\}. 
\end{align*}
Let $i=1$ or $2$. We are concerned with asymptotic behaviours of the solution $f_\ep=f_\ep(t,x,v)$, $(t,x,v)\in\R_+\times\O\times\R^d$, to the following dissipative kinetic equation with the parameter $\ep\in(0,1]$,  
\begin{align}\label{rfp}
\left\{ 
\begin{aligned}
\ &\ep\p_t f_\ep+v\cdot\na_xf_\ep -\na_x\phi\cdot\na_vf_\ep =\ep^{-1}\LL_i f_\ep \quad {\rm in\ } \R_+\times\O\times\R^d,  \\
\ &\,f_\ep|_{t=0} =\fein \quad {\rm in\ } \O\times\R^d, \\
\ &\,f_\ep=\alpha\DD f_\ep +\beta\RR f_\ep  {\quad\rm in\ }\R_+\times\Sigma_-. \\
\end{aligned}
\right. 
\end{align} 
Here the external potential $\phi\in C^{0,1}(\O)$, and the (space-dependent) coefficients $\alpha,\beta:\p\O\rightarrow[0,1]$ in the boundary condition are measurable and satisfy $\alpha+\beta\le1$ and $\alpha+\beta\in C^{0,1}(\p\O)$. For $f(x,v)\in C_c^\infty(\Sigma)$, the specular reflection operator $\RR$ is defined by 
\begin{align*}
\RR f(x,v):=f(x,v-2(n_x\cdot v)\;\!n_x), 
\end{align*}
and the diffuse reflection operator $\DD$ is defined by 
\begin{align*}
\DD f(x):=c_w \int_{\R^d}f(x,v)\,(n_x\cdot v)_+\dif\mu,
\end{align*}
where the constant $c_w:=\sqrt{2\pi}$ is chosen so that $c_w\int_{\R^d}(n_x\cdot v)_\pm\dif\mu=1$. Note that $(n_x\cdot v)_+$ and $(n_x\cdot v)_-$ denote the positive and negative parts of the function $n_x\cdot v$, respectively. For brevity, we also set 
\begin{align*}
\dif m:=e^{-\phi(x)}\mu(v)\dif x\dif v,\\
\|\cdot\|:=\|\cdot\|_{L^2(\O\times\R^d,\dif m)}. 
\end{align*}
Throughout the article, we assume that the solutions mentioned, which exist in bounded spatial domains in the weak sense, allow us the freedom to manipulate boundary terms. Indeed, for any initial data  $\fein\in L^2(\O\times\R^d,\dif m)$, there is a solution $f\in C^0(\R_+;L^2(\O\times\R^d,\dif m))$ satisfying \eqref{rfp} in the sense of distribution that admits suitable traces on $\p\O\times\R^d$. One may refer to \cite{Mischler,DHHM} for rigorous details in this regard. 

We aim in this work at investigating the long-time asymptotics of the solution $f_\ep$ to \eqref{rfp} and its limit as $\ep\rightarrow0$. When $\alpha+\beta=1$, \eqref{rfp} is associated with the so-called \emph{Maxwell boundary condition} which conserves the mass. Generally, \eqref{rfp} is not conservative. Our first result gives hypocoercive estimates with uniform-in-$\ep$ exponential decay rates for solutions to \eqref{rfp} associated with general boundary conditions. 

\begin{theorem}\label{longtime}
Let $\ep\in(0,1]$, $\fein\in L^2(\O\times\R^d,\dif m)$, $M_0:=\left(\int_\O e^{-\phi}\dif x\right)^{\!-1} \int_{\O\times\R^d} \fein\dif m$, and the functions $\alpha,\beta:\p\O\rightarrow[0,1]$ satisfy $\alpha+\beta\le1$ and $\alpha+\beta\in C^{0,1}(\p\O)$, and $f_\ep$ be the solution to \eqref{rfp}. If $\alpha+\beta=1$, then there are some constants $C,\lambda>0$ depending only on $d,\O,\|\phi\|_{C^{0,1}(\O)}$ such that for any $t>0$, 
\begin{align*}
\|f_\ep(t)-M_0\|  \le C e^{-\lambda t}\|\fein-M_0\|. 
\end{align*}
If $\alpha+\beta\in[0,\delta]$ for some constant $\delta\in[0,1)$, then there are some constants $C',\lambda'>0$ depending only on $d,\O,\|\phi\|_{C^{0,1}(\O)},\delta,\|\na(\alpha+\beta)\|_{L^\infty(\p\O)}$ such that for any $t>0$, 
\begin{align*}
\|f_\ep(t)\| \le C' e^{-\lambda't}\|\fein\|.  
\end{align*}
\end{theorem}

The second result shows that the asymptotic dynamics of \eqref{rfp} (with $\beta=1-\alpha$) as $\ep\rightarrow0$ is governed by the following parabolic equation associated with the Neumann boundary condition, 
\begin{align}\label{limiteq}
\left\{ 
\begin{aligned}
\ &\p_t\rho=\Delta_x\rho -\na_x\phi\cdot\na_x\rho {\quad\rm in\ } \R_+\times\O, \\
\ &\, \rho|_{t=0} = \rin  {\quad\rm in\ } \O,\\
\ &\, n_x\cdot\na_x\rho=0 {\quad\rm in\ } \R_+\times\p\O.\\
\end{aligned}
\right. 
\end{align} 

\begin{theorem}\label{asympto}
Let $\ep\in(0,1]$, $\fein\in L^2(\O\times\R^d,\dif m)$, and $\alpha:\p\O\rightarrow[0,1]$. Suppose that $f_\ep$ is the solution to \eqref{rfp} (with $\beta=1-\alpha$) associated with the Maxwell boundary condition $f_\ep=\alpha\DD f_\ep+(1-\alpha)\RR f_\ep$ on $\R_+\times\Sigma_-$, and $\rho$ is the solution to \eqref{limiteq} associated with the initial data $\rin:=\int_{\R^d}\fein\dif\mu$. Then we have the strong convergence as $\ep\rightarrow0$ that for any $T>0$ and any positive function $\tau$ defined on $(0,1]$ such that $\tau(\ep)/\ep^2\rightarrow\infty$ as $\ep\rightarrow0$, 
\begin{align*}
\|f_\ep-\rho\|_{L^2([0,T]\times\O\times\R^d,\;\!\dif t\dif m)} \rightarrow 0,\\ 
\|f_\ep-\rho\|_{L^\infty([\tau(\ep),\infty);L^2(\O\times\R^d,\;\!\dif m))} \rightarrow 0.
\end{align*}  
If $\fein\in L_v^2H_x^2(\O\times\R^d,\dif m)$ satisfies $|v|\na_x\fein\in L^2(\O\times\R^d,\dif m)$ and the compatibility condition that $\fein$ depends only on the variable $x$ in $\p\O\times\R^d$, then for any $T>0$ and for $T_\ep:=\ep^2|\log\sqrt{\ep}|$, 
\begin{align*}
\|f_\ep-\rho\|_{L^2([0,T]\times\O\times\R^d,\;\!\dif t\dif m)} 
\le C \sqrt{\ep T} \left(\|\fein\| +\|(1+|v|)\na_x\fein\| +\|D_x^2\fein\| \right),\\
\|f_\ep-\rho\|_{L^\infty([T_\ep,\infty);L^2(\O\times\R^d,\;\!\dif m))} 
\le C \sqrt{\ep} \left(\|\fein\| +\|(1+|v|)\na_x\fein\| +\|D_x^2\fein\| \right). 
\end{align*}
If $\fein$ is well-prepared in the sense that $\fein\in H_x^2(\O)$ depends only on the variable $x$ in $\O\times\R^d$, then 
\begin{align*}
\|f_\ep-\rho\|_{L^\infty(\R_+;\;\!L^2(\O\times\R^d,\;\!\dif m))} 
\le C\sqrt{\ep} \|\fein\|_{H_x^2(\O)}. 
\end{align*}
Here the constant $C>0$ depends only on $d,\O,\|\phi\|_{C^{0,1}(\O)}$. 
\end{theorem}

The model \eqref{rfp}, with $i=1$ or $2$, is a kinetic description of the probability distribution of a certain system of interacting particles, submitted to an external force derived from the potential $\phi$, at time $t$ located at the position $x$ in the physical space $\O\subset\R^d$ with the velocity $v\in\R^d$. It is complemented with the proper boundary conditions corresponding to the behaviour of particles at the boundary. The boundary condition in \eqref{rfp} is prescribed as zero influx when $\alpha=\beta=0$, and it is generally governed by the balance relations between the distribution of particles at the incoming and outgoing boundaries. The operator $\LL_1$ serves as a prototype for describing collision processes of linear relaxation type, including the case of neutron transport. The operator $\LL_2$ derived from the Ornstein-Uhlenbeck velocity process captures the impact of collisions with  particles of a surrounding bath. The collision operators act only on the velocity variable and conserve local mass. They are both dissipative in the velocity variable and contribute to the relaxation process towards local equilibrium. The \emph{hypocoercivity} theory seeks to understand how the interaction between transport and collisions leads to time-decay convergence for inhomogeneous collisional kinetic equations such as \eqref{rfp}. The (small) parameter $\ep\in(0,1]$ representing the ratio of the collisional mean free path and the observation length is introduced to measure the balance between the transport part $v\cdot\na_x$ and the collision part $\LL_i$, while diffusion phenomena can be observed on long time scales of order $\ep^{-1}$. This regime with $\ep\rightarrow0$ is often referred to as the \emph{diffusion limit}. 

In the absence of boundary conditions, the systematic development of $H^1$-hypocoercivity for kinetic diffusion equations (such as \eqref{rfp} with $i=2$) can be found in \cite{Villanihypo}, which typically involves high order regularity estimates for solutions. Nevertheless loss of regularity may occur near the grazing boundaries; see \cite[Appendix A]{YZ}. The $L^2$-hypocoercivity instead provides $L^2$-decay for a large class of  dissipative kinetic equations without assuming regularity on the solution; see for instance \cite{Herau,DMS1,EGKM,DMS2,BDMMS}, where the spatial domain is either the periodic box $\mathbb{T}^d$ or the whole space $\R^d$. The approaches given in \cite{BCMT} and \cite{DHHM} yield $L^2$-decay estimates when considering boundary conditions that conserve mass. In \cite{DD}, $L^2$-decay estimates were established for a specific scenario involving exponentially decaying influx boundary conditions. We also point out that, for kinetic Fokker-Planck equations, constructive exponential convergence result was given in \cite{AM2021} by means of a Poincar\'e-type inequality under the zero influx boundary condition, and nonconstructive estimates in a general non-conservative setting were given in \cite{FGM} by revisiting the Krein-Rutman theory. 

The mathematical study of approximating linear kinetic equations through macroscopic diffusion equations has been investigated in previous works such as \cite{LK,BLP,BSS,Degond}. It was noticed in \cite{WG} that the presence of grazing boundaries may cause a breakdown in the diffusive asymptotics within the $L^\infty$ framework. We refer to \cite{Wu3,Ouyang} for recent development of the diffusion limit of the neutron transport equation associated with boundary conditions. In addition to the analysis through formal expansions as employed in the aforementioned works, the relative entropy method has also emerged as a valuable approach in addressing the hydrodynamic asymptotics; see for instance \cite{BGL,LSR}. 

Our study draws inspiration from the $L^2$-hypocoercivity and relative entropy methods used in \cite{AZ} which focused on a nonlinear kinetic Fokker-Planck model with the spatial domain of either $\mathbb{T}^d$ or $\R^d$. By making simple observations on the boundary terms and employing basic energy estimates, we are able to derive quantitative $L^2$-hypocoercive estimates uniformly in $\ep\in(0,1]$ and $L^2$-diffusion asymptotics as $\ep\rightarrow0$ for solutions to \eqref{rfp} in the presence of boundary conditions. 

The article is organized as follows. In Section~\ref{pre}, we introduce some notations and present preliminary estimates related to our results. We prove Theorem~\ref{longtime} in Section~\ref{hypocoercive} and Theorem~\ref{asympto} in Section~\ref{diffusionlimit}. Some basic elliptic estimates are recalled in Appendix~\ref{append}. 

\section{Preliminaries}\label{pre} 
\subsection{Notations} 
Let us introduce the notations that will be used. For any couple of (scalar, vector or $d\times d$-matrix valued) functions $\Psi_1,\Psi_2\in L^2\left(\O\times\R^d,\dif m\right)$, we denote their inner product with respect to the measure $\dif m=e^{-\phi}\dif x\dif\mu$ by 
\begin{align*}
\left(\Psi_1,\, \Psi_2\right) :=\int_{\O\times\R^d} \Psi_1\Psi_2\dif m, 
\end{align*}
where the multiplication between the couple in the integrand should be replaced by scalar contraction product, if $\Psi_1,\Psi_2$ is a couple of vectors or matrices. Then $\|\cdot\|$ is its induced norm. Since the domain $\O$ is bounded and the potential function $\phi$ is Lipschitz, the norm $\|\cdot\| $ is equivalent to the norm $\|\cdot\|_{L^2(\O\times\R^d,\dif x\dif\mu)}$. Similarly, with $\dif\sigma_x$ denoting the surface measure on $\p\O$, the inner product on the phase boundary $\p\O\times\R^d$ for $\Psi_1,\Psi_2\in L^2\left(\p\O\times\R^d,e^{-\phi}\dif\sigma_x\dif\mu\right)$ is defined by 
\begin{align*}
\left(\Psi_1,\, \Psi_2\right)_\p:=\int_{\p\O\times\R^d} \Psi_1\Psi_2\,e^{-\phi}\dif\sigma_x\dif \mu. 
\end{align*}
Its induced norm is denoted by $\|\cdot\|_\p$. In this setting, we have 
\begin{align*}
\left(\psi_1,\,(\na_v-v)\psi_2\right) = -\left(\na_v\psi_1,\,\psi_2\right), 
\end{align*}
and the transport operator $v\cdot\na_x-\na_x\phi\cdot\na_v$ satisfies 
\begin{align*}
\left((v\cdot\na_x-\na_x\phi\cdot\na_v)\psi_1,\,\psi_2\right) 
+\left(\psi_1,\, (v\cdot\na_x-\na_x\phi\cdot\na_v)\psi_2\right) 
=\left(\psi_1\psi_2,\, n_x\cdot v\right)_\p, 
\end{align*}
for any scalar functions $\psi_1,\psi_2\in C_c^\infty(\overline{\O}\times\R^d)$. 

Note that if the functions $\Psi_1=\Psi_1(x)$, $\Psi_2=\Psi_2(x)\in C^\infty(\overline{\O})$ are independent the velocity variable, then the inner products simplify to
\begin{align*}
\left(\Psi_1,\, \Psi_2\right) &=\int_{\O} \Psi_1\Psi_2\,e^{-\phi}\dif x, \\
\left(\Psi_1,\, \Psi_2\right)_\p&=\int_{\p\O} \Psi_1\Psi_2\,e^{-\phi}\dif\sigma_x. 
\end{align*}
In particular, for the velocity-independent function $\Psi=\Psi(x)\in C^\infty(\overline{\O})$, the norms simplify to $\|\Psi\|^2=\int_\Omega|\Psi|^2e^{-\phi}\dif x$ and $\|\Psi\|_\p^2=\int_{\p\Omega}|\Psi|^2e^{-\phi}\dif\sigma_x$. These notations are used for brevity.

For $\Psi\in L^1(\R^d,\dif\mu)$, we denote the projection onto Span$\{\mu\}$ and its orthogonal complement by 
\begin{align*}
\la\Psi\ra:=\int_{\R^d} \Psi\dif\mu,\\
\Psi^\perp:=\Psi-\la\Psi\ra.
\end{align*}
Taking the bracket $\la\cdot\ra$ after multiplying the equation in \eqref{rfp} with $1$ and $v$ leads to the following macroscopic equations, 
\begin{align}
&\ep\p_t\la f_\ep\ra+\na_x\cdot\la vf_\ep\ra -\na_x\phi\cdot\la vf_\ep\ra = 0,  \label{hydroFP1} \\
&\ep\p_t\la vf_\ep\ra + \na_x\cdot\la v\otimes vf_\ep\ra -\na_x\phi\cdot\la(v\otimes v-{\rm I})f_\ep\ra = -\ep^{-1}\la vf_\ep\ra.  \label{hydroFP2}
\end{align}

Throughout the rest of the article, the notation $X\lesssim Y$ means that $X\le CY$ for some constant $C>0$ depending only on $d,\O,\|\phi\|_{C^{0,1}(\O)},\delta,\|\na(\alpha+\beta)\|_{L^\infty(\p\O)}$. 

\subsection{Boundary estimates} 
We start by presenting the following two identities, which are applicable to functions subject to general boundary conditions and will play a crucial role in handling the boundary terms in the later estimates. 

\begin{lemma}\label{identity}
Let the functions $\alpha,\beta:\p\O\rightarrow[0,1]$ satisfy $\alpha+\beta\le1$, and the vector $U:\p\O\rightarrow\R^d$ lying in $L^2(\p\O)$ satisfy 
\begin{align*}
n_x\cdot U=0.
\end{align*}
For any function $f:\Sigma\rightarrow\R$ such that $f=\alpha\DD f+\beta\RR f$ on $\Sigma_-$, we have 
\begin{align}
\left(v\cdot U f,\,n_x\cdot v\right)_\p = \left(v\cdot U [(1-\beta)f-\alpha\DD f],\,(n_x\cdot v)_+\right)_\p, \label{identity1}
\end{align} 
and
\begin{align}
\begin{aligned}
\left(f^2,\,n_x\cdot v\right)_\p 
&=\left((1-\beta^2)(f-\DD f)^2,\,(n_x\cdot v)_+\right)_\p + \left((1-(\alpha+\beta)^2)(\DD f)^2,\,(n_x\cdot v)_+\right)_\p \label{identity2} \\
&= \left((\alpha^2+2\alpha\beta)(f-\DD f)^2,\,(n_x\cdot v)_+\right)_\p + \left((1-(\alpha+\beta)^2)f^2,\,(n_x\cdot v)_+\right)_\p. 
\end{aligned}
\end{align} 
\end{lemma}

\begin{proof}
According to the assumption on $f$, 
\begin{align*}
\left(v\cdot Uf,\,n_x\cdot v\right)_\p
=\left(v\cdot Uf,\,(n_x\cdot v)_+\right)_\p -\left(v\cdot U \alpha\DD f,\,(n_x\cdot v)_-\right)_\p
-\left(v\cdot U\beta\RR f,\,(n_x\cdot v)_-\right)_\p. 
\end{align*}
On the one hand, since the vector field $U\alpha\DD f$ is independent of $v$ and tangent to $n_x$, we see that for any fixed $x\in\p\O$, the functions $v\mapsto v\cdot U \alpha\DD f\,(n_x\cdot v)_\pm$ are odd with respect to the tangential component of $v$ at $\p\O$. This implies that 
\begin{align*}
\left(v\cdot U \alpha\DD f,\,(n_x\cdot v)_\pm\right)_\p  =0. 
\end{align*}
On the other hand, by a change of variables $v\mapsto v-2(n_x\cdot v)\;\!n_x$ and the assumption on $U$, we have 
\begin{align*}
\left(v\cdot U\beta\RR f,\,(n_x\cdot v)_-\right)_\p 
&=\left(v\cdot U\beta f,\,(n_x\cdot v)_+\right)_\p
-\left(2(n_x\cdot v)n_x\cdot U\beta f,\,(n_x\cdot v)_+\right)_\p\\
&=\left(v\cdot U\beta f,\,(n_x\cdot v)_+\right)_\p. 
\end{align*}
The three identities above together establish \eqref{identity1}. 

As for \eqref{identity2}, we observe that for any bounded function $a_b$ defined on $\p\O$ (we will take $a_b$ to be either $1-\beta^2$ or $\alpha^2+2\alpha\beta$), we have 
\begin{align*}
(a_bf\DD f,\,(n_x\cdot v)_+)_\p=(a_b(\DD f)^2,\,(n_x\cdot v)_+)_\p,
\end{align*}
which implies that 
\begin{align*}
\left(a_b(f-\DD f)^2,\,(n_x\cdot v)_+\right)_\p 
=\left(a_bf^2,\,(n_x\cdot v)_+\right)_\p - \left(a_b(\DD f)^2,\,(n_x\cdot v)_+\right)_\p. 
\end{align*}
By definition and changes of variables, we have 
\begin{align*}
\left(f^2,\,n_x\cdot v\right)_\p 
= \left(f^2,\,(n_x\cdot v)_+\right)_\p -\left(\alpha^2(\DD f)^2,\,(n_x\cdot v)_-\right)_\p \\ 
-\left(2\alpha\beta\DD f\RR f,\,(n_x\cdot v)_-\right)_\p -\left(\beta^2(\RR f)^2,\,(n_x\cdot v)_-\right)_\p \\ 
= \left(f^2,\,(n_x\cdot v)_+\right)_\p -\left(\alpha^2(\DD f)^2,\,(n_x\cdot v)_+\right)_\p 
-\left(2\alpha\beta f\DD f,\,(n_x\cdot v)_+\right)_\p -\left(\beta^2 f^2,\,(n_x\cdot v)_+\right)_\p \\ 
=\left((1-\beta^2)f^2,\,(n_x\cdot v)_+\right)_\p -  \left((\alpha^2+2\alpha\beta)(\DD f)^2,\,(n_x\cdot v)_+\right)_\p&. 
\end{align*}
The proof of \eqref{identity2} is complete by gathering the above two identities. 
\end{proof}

\begin{lemma}\label{identity-e}
Let the functions $\alpha,\beta:\p\O\rightarrow[0,1]$ satisfy $\alpha+\beta\le1$, and the vector $U:\p\O\rightarrow\R^d$ lying in $L^2(\p\O)$ satisfy 
\begin{align*}
n_x\cdot U=0. 
\end{align*}
Suppose that the function
\begin{align*}
c_b:=(1-\alpha-\beta)/(\alpha+\beta)\in L^\infty(\p\O).
\end{align*} 
For any functions $\varrho\in L^2(\p\O)$ and $f:\Sigma\rightarrow\R$ such that $f=\alpha\DD f+\beta\RR f$ on $\Sigma_-$, we have 
\begin{align*}
\left|\left(v\cdot(U-c_b\varrho n_x) f,\,n_x\cdot v\right)_\p\right|
\lesssim \left(\|U-c_b\varrho n_x\|_\p + \|\sqrt{c_b}\varrho\|_\p\right) \left(f^2,\, n_x\cdot v\right)_\p^\frac{1}{2}. 
\end{align*}
\end{lemma}

\begin{proof}
By the assumption on $f$ and a change of variables, 
\begin{align*}
\left(c_b\varrho f,\,(n_x\cdot v)^2\right)_\p=\left(c_b\varrho[(1+\beta)f+\alpha\DD f],\,(n_x\cdot v)_+^2\right)_\p.
\end{align*} 
In view of \eqref{identity1} of Lemma~\ref{identity}, we obtain 
\begin{align}\label{lemma-bdy1}
\begin{aligned}
\left(v\cdot(U-c_b\varrho n_x)f,\,n_x\cdot v\right)_\p
=\left(v\cdot U[(1-\beta)f-\alpha\DD f],\,(n_x\cdot v)_+\right)_\p
- \left(c_b\varrho f,\,(n_x\cdot v)^2\right)_\p \\
=\left(v\cdot(U-c_b\varrho n_x)[(1-\beta)f-\alpha\DD f],\,(n_x\cdot v)_+\right)_\p 
-\big(c_b\varrho(\beta f+\alpha\DD f),\,(n_x\cdot v)_+^2\big)_\p&. 
\end{aligned}
\end{align}
By using H\"older's inequality, we have 
\begin{align}\label{lemma-bdy2}
\begin{aligned}
\left|\left(v\cdot(U-c_b\varrho n_x)[(1-\beta)f-\alpha\DD f],\,(n_x\cdot v)_+\right)_\p\right| \\
\le \left(|v|^2|U-c_b\varrho n_x|^2,\,(n_x\cdot v)_+\right)_\p^\frac{1}{2} \left((1-\beta)^2(f-\DD f)^2 +(1-\alpha-\beta)^2(\DD f)^2,\, (n_x\cdot v)_+ \right)_\p^\frac{1}{2}&, 
\end{aligned}
\end{align}
and along with the fact that $c_b(\alpha^2+\beta^2)\le 1-\alpha-\beta$, we have 
\begin{align}\label{lemma-bdy3}
\begin{aligned}
\left|\left(c_b\varrho(\beta f+\alpha\DD f),\,(n_x\cdot v)_+^2\right)_\p\right| 
\le \left(c_b\varrho^2,\,(n_x\cdot v)_+^3\right)_\p^\frac{1}{2} \left(2c_b\big[\beta^2 f^2 +\alpha^2(\DD f)^2\big],\, (n_x\cdot v)_+ \right)_\p^\frac{1}{2}\\
\le 2\left(c_b\varrho^2,\,(n_x\cdot v)_+^3\right)_\p^\frac{1}{2} \left((1-\alpha-\beta)\big[f^2 +(\DD f)^2\big],\, (n_x\cdot v)_+ \right)_\p^\frac{1}{2}&. 
\end{aligned}
\end{align}
Recalling that the inner product $(\cdot,\cdot)_\p$ incorporates a Gaussian weight, and noting that the functions $U-c_b\varrho n_x$ and $\sqrt{c_b}\varrho$ are independent of $v$, we see that 
\begin{align}\label{lemma-bdy41}
\begin{aligned}
\left(|v|^2|U-c_b\varrho n_x|^2,\,(n_x\cdot v)_+\right)_\p^\frac{1}{2} \lesssim \|U-c_b\varrho n_x\|_\p, \\
\left(c_b\varrho^2,\,(n_x\cdot v)_+^3\right)_\p^\frac{1}{2} \lesssim \|\sqrt{c_b}\varrho\|_\p.
\end{aligned}
\end{align}
We thus conclude from \eqref{lemma-bdy1}, \eqref{lemma-bdy2}, \eqref{lemma-bdy3}, \eqref{lemma-bdy41} that
\begin{align}\label{lemma-bdy-f1}
\begin{aligned}
\left|\left(v\cdot(U-c_b\varrho n_x) g,\,n_x\cdot v\right)_\p\right|\\
\lesssim \|U-c_b\varrho n_x\|_\p \left((1-\beta)^2(f-\DD f)^2 +(1-\alpha-\beta)^2(\DD f)^2,\, (n_x\cdot v)_+ \right)_\p^\frac{1}{2} \\ 
+ \|\sqrt{c_b}\varrho\|_\p\left((1-\alpha-\beta)\big[f^2+(\DD f)^2\big],\,(n_x\cdot v)_+\right)_\p^\frac{1}{2} &. 
\end{aligned}
\end{align}
For $\alpha,\beta\ge0$ such that $\alpha+\beta\le1$, we have the following elementary inequalities, 
\begin{align*}
(1-\beta)^2\le1-\beta^2,\\
(1-\alpha-\beta)^2\le 1-\alpha-\beta \le1-(\alpha+\beta)^2. 
\end{align*}
Taking these facts into account and applying \eqref{identity2} of Lemma~\ref{identity}, we have  
\begin{align}\label{lemma-bdy-f2}
\begin{aligned}
\left((1-\beta)^2(f-\DD f)^2 +(1-\alpha-\beta)^2(\DD f)^2,\, (n_x\cdot v)_+ \right)_\p\\
\le \left((1-\beta^2)(f-\DD f)^2 +(1-(\alpha+\beta)^2)(\DD f)^2,\, (n_x\cdot v)_+ \right)_\p \\
= \left(f^2,\,n_x\cdot v\right)_\p&, 
\end{aligned}
\end{align}
and 
\begin{align}\label{lemma-bdy-f3}
\begin{aligned}
\left((1-\alpha-\beta)\big[f^2+(\DD f)^2\big],\,(n_x\cdot v)_+\right)_\p\\
\le \left((1-(\alpha+\beta)^2)\big[f^2+(\DD f)^2\big],\,(n_x\cdot v)_+\right)_\p\\
\le 2 \left(f^2,\,n_x\cdot v\right)_\p&. 
\end{aligned}
\end{align}
Gathering \eqref{lemma-bdy-f1}, \eqref{lemma-bdy-f2} and \eqref{lemma-bdy-f3}, we arrive at the desired result. 
\end{proof}

\subsection{Dissipation estimates} 
The dissipation property exhibited by the operator $\LL_i$, which acts only on the velocity variable, is formulated in the following lemma. 

\begin{lemma}\label{dissipation}
If $f_\ep$ is a solution to \eqref{rfp} associated with the initial data $\fein\in L^2(\O\times\R^d,\dif m)$ and the boundary condition $f_\ep=\alpha\DD f_\ep +\beta\RR f_\ep$ on $\R_+\times\Sigma_-$, then we have 
\begin{align}\label{dissipation-h}
\frac{\dif}{\dif t} \|f_\ep\|^2  
\le -\frac{2}{\ep^2}\|f_\ep^\perp\|^2- \frac{1}{\ep}\left(f_\ep^2,\, n_x\cdot v\right)_\p,  
\end{align}
with $\left(f_\ep^2,\, n_x\cdot v\right)_\p\ge0$. In particular, for any $t>0$, 
\begin{align}\label{dissipationR}
\|f_\ep(t)\|^2 +\frac{2}{\ep^2} \int_0^t \|f_\ep^\perp\|^2
+\frac{1}{\ep}\int_0^t\left(f_\ep^2,\,n_x\cdot v\right)_\p
\le \|\fein\|^2. 
\end{align}
\end{lemma}

\begin{proof}
A direct computation by integrating \eqref{rfp} against $\mu e^{-\phi}f_\ep$ yields that 
\begin{align*}
\frac{1}{2}\frac{\dif}{\dif t}\|f_\ep\|^2
&=\left(f_\ep,\,\p_t f_\ep\right) \\
&=\frac{1}{\ep^2} \left(\LL_if_\ep,f_\ep\right) 
-\frac{1}{2\ep} \left(f_\ep^2,\,n_x\cdot v\right)_\p\\
&\le -\frac{1}{\ep^2}\|f_\ep^\perp\|^2 -\frac{1}{2\ep} \left(f_\ep^2,\,n_x\cdot v\right)_\p. 
\end{align*}
Indeed, we used the definition of $\LL_i$ and the Gaussian-Poincar\'e inequality with $\la f_\ep^\perp\ra=0$ to see that 
\begin{align*}
\left(\LL_1f_\ep,f_\ep\right) &= -\|f_\ep^\perp\|^2, \\
\left(\LL_2f_\ep,f_\ep\right) &= -\|\na_vf_\ep^\perp\|^2 \le -\|f_\ep^\perp\|^2.
\end{align*}
Here we point out that $\left(f_\ep^2,\,n_x\cdot v\right)_\p\ge0$ due to \eqref{identity2} of Lemma~\ref{identity}. Integrating along time, we derive \eqref{dissipationR} as claimed. 
\end{proof}

The following lemma gives a weighted dissipation estimate for solutions to the homogeneous counterpart of \eqref{rfp}, which will be used in Section~\ref{diffusionlimit} to address the initial layer correction. It is worth noting that the variable $x$ involved in the lemma below serves as a parameter and therefore does not affect the relaxation process. 

\begin{lemma}\label{layer}
There exists some constant $c_0>0$ such that for any solution $\psi_\ep$ to
\begin{align}\label{layereq}
\left\{ 
\begin{aligned}
\ &\p_t\psi_\ep=\ep^{-2}\LL_i\psi_\ep  \quad {\rm in\ } \R_+\times\O\times\R^d, \\
\ &\psi_\ep|_{t=0}=\psin \quad {\rm in\ } \O\times\R^d, \\
\end{aligned}
\right. 
\end{align} 
with the initial data $\psin$ satisfying $(1+|v|)\psin\in L^2(\R^d,\dif\mu)$ and $\la\psin\ra=0$, we have 
\begin{align}
\|\psi_\ep(t)\|^2 +\frac{2}{\ep^2} \int_0^t \|\psi_\ep\|^2 \le \|\psin\|^2, \label{layerdecay1} \\	
\|(1+|v|)\psi_\ep(t)\| \lesssim e^{-c_0t/\ep^2}\|(1+|v|)\psin\|. \label{layerdecay2}
\end{align}
\end{lemma}

\begin{proof}
We first observe that $\la\psi_\ep(t)\ra=0$ for any $t\ge0$. One is thus able to arrive at \eqref{layerdecay1} in the same way as the derivation of \eqref{dissipationR} in Lemma~\ref{dissipation}. 

In order to show the weighted estimate \eqref{layerdecay2} for solutions to \eqref{layereq} with $i=1$ (the time-relaxation case), we integrate \eqref{layereq} against $\mu e^{-\phi}\psi_\ep$ and $\mu e^{-\phi}|v|^2\psi_\ep$ so that 
\begin{align*}
\frac{\dif}{\dif t}\|\psi_\ep\|^2= -\frac{2}{\ep^2}\|\psi_\ep\|^2,\\
\frac{\dif}{\dif t}\|v\psi_\ep\|^2 = -\frac{2}{\ep^2}\|v\psi_\ep\|^2, 
\end{align*}
which then implies \eqref{layerdecay2} from Gr\"onwall's inequality directly. 

As for \eqref{layereq} with $i=2$ (the Fokker-Planck case), we apply integration by parts, Gaussian-Poincar\'e inequality and Cauchy-Schwarz inequality so that 
\begin{align}\label{re-psi1}
\begin{aligned}
\frac{\dif}{\dif t}\|\psi_\ep\|^2
=-\frac{2}{\ep^2}\|\na_v\psi_\ep\|^2
\le -\frac{1}{\ep^2}\left(\|\psi_\ep\|^2+\|\na_v\psi_\ep\|^2\right),
\end{aligned}
\end{align}
and
\begin{align}\label{re-psi2}
\begin{aligned}
\frac{\dif}{\dif t}\|v\psi_\ep\|^2
=-\frac{2}{\ep^2}\int_{\O\times\R^d}\left(|v|^2|\na_v\psi_\ep|^2+2\psi_\ep v\cdot\na_v\psi_\ep\right)\dif m
\le \frac{2}{\ep^2}\|\psi_\ep\|^2.  
\end{aligned}
\end{align}
Recall that for any $\psi\in L^2(\R^d,\dif\mu)$ such that $\na_v\psi\in L^2(\R^d,\dif\mu)$, we have 
\begin{align*}
\int_{\R^d}|v\psi|^2\dif\mu
&=-2\int_{\R^d}\psi^2\,v\cdot\na_v\mu\dif v-\int_{\R^d}|v\psi|^2\dif\mu\\
&=2d\int_{\R^d}\psi^2\dif\mu + 4\int_{\R^d}\psi\,v\cdot\na_v\psi\dif\mu-\int_{\R^d}|v\psi|^2\dif\mu\\
&\le 2d\int_{\R^d}\psi^2\dif\mu + 4\int_{\R^d}|\na_v\psi|^2\dif\mu, 
\end{align*}
where we used the integration by parts and the Cauchy-Schwarz inequality. 
We thus obtain 
\begin{align*}
\|v\psi_\ep\|\lesssim\|\na_v\psi_\ep\|+\|\psi_\ep\|. 
\end{align*}
Combining this with \eqref{re-psi1}, \eqref{re-psi2} and Gr\"onwall's inequality, we then derive \eqref{layerdecay2} as claimed. 
\end{proof}

\section{Long-time asymptotics}\label{hypocoercive} 
This section is devoted to the proof of Theorem~\ref{longtime}. We consider dissipative kinetic equations in bounded spatial domains with general boundary conditions that may not conserve mass, and discuss its hypocoercive property concerning the (uniform-in-$\ep$) long-time behaviour of the solutions. We recall the initial mass $M_0:=\left(\int_\O e^{-\phi}\dif x\right)^{-1} \int_{\O\times\R^d} \fein\dif m$. It can be readily verified that the mass conservation of \eqref{rfp} holds when $\alpha+\beta=1$. 

Let us turn to the proof of Theorem~\ref{longtime}. 

\begin{proof}[Proof of Theorem~\ref{longtime}]
The proof will proceed in three steps. 
	
\medskip\noindent\textit{Step 1. Elliptic problem associated with macroscopic quantities. }\\	
We consider the solution $u$ to the following elliptic equation with the Robin boundary condition, 
\begin{align}\label{elliptic-u}
\left\{ 
\begin{aligned}
\ & u-\Delta_x u+ \na_x\phi\cdot\na_xu=\la f_\ep\ra-M_c {\quad\rm in\ }\O, \\
\ & n_x\cdot\na_xu +c_bu =0 {\quad\rm on\ }\p\O. \\
\end{aligned}
\right. 
\end{align} 
Here we set $c_b:=(1-\alpha-\beta)/(\alpha+\beta)$ (we may always assume $\alpha+\beta\ge\iota>0$ so that $c_b$ is bounded, and eventually let $\iota\rightarrow0$); we take $M_c:=0$ if $c_b|_{\p\O}>0$, and $M_c:=M_0$ if $c_b|_{\p\O}=0$. In particular, when $c_b|_{\p\O}=0$, that is, $\alpha+\beta=1$ on $\p\O$, we deduce from \eqref{rfp} that the mass conservation $\left(\la f_\ep\ra-M_c,\,1\right)=0$ is satisfied. 
Let us consider the quantity 
\begin{align*}
\mathcal{A}:=\left(\la f_\ep\ra-M_c-u,\,\la f_\ep\ra-M_c\right).  
\end{align*}
By \eqref{elliptic-z} of Lemma~\ref{elliptic-est} (see Appendix~\ref{append} below), we obtain 
\begin{align}\label{AA}
\|\la f_\ep\ra-M_c\| + \|u\| + \|\na_xu\| + \|D_x^2u\| + \|\sqrt{c_b}u\|_\p \lesssim \mathcal{A}^\frac{1}{2}. 
\end{align} 
Since $\la vf_\ep\ra$=$\la vf_\ep^\perp\ra$, the local conservation law \eqref{hydroFP1} can be written as  
\begin{align*}
\ep\p_t\la f_\ep\ra+\na_x\cdot\la vf_\ep^\perp\ra-\na_x\phi\cdot\la vf_\ep^\perp\ra=0. 
\end{align*}
Combining this relation with \eqref{elliptic-u} and applying integration by parts, we have
\begin{align}\label{elliptic1}
\begin{aligned}
\|\p_tu\|^2 +\|\na_x\p_tu\|^2 + \|\sqrt{c_b}\p_tu\|_\p^2
&= \left( \p_tu,\,\p_t(u-\Delta_xu+\na_x\phi\cdot\na_xu)\right)\\
&= \left( \p_tu,\,\p_t\la f_\ep\ra\right) \\
&= \ep^{-1}\left( \p_tu,\,(\na_x\phi-\na_x)\cdot\la vf_\ep^\perp\ra\right) \\
&= \ep^{-1}\left( \na_x\p_tu,\,\la vf_\ep^\perp\ra\right) - \ep^{-1}\left( \p_tu,\,n_x\cdot\la vf_\ep^\perp\ra\right)_\p. 
\end{aligned} 
\end{align} 
By H\"older's inequality, 
\begin{align}\label{elliptic2}
\begin{aligned}
\left|\left(\na_x\p_tu,\,\la vf_\ep^\perp\ra\right)\right| 
&\le \|v\cdot\na_x\p_tu\| \;\! \|f_\ep^\perp\| \\
&= \|\na_x\p_tu\| \;\! \|f_\ep^\perp\|. 
\end{aligned} 
\end{align} 
Based on the boundary condition of $f_\ep$, the boundary term above can be recast as
\begin{align*}
\left(\p_tu,\,n_x\cdot\la vf_\ep^\perp\ra\right)_\p = \left( \p_tu\;\!(1-\alpha-\beta)(f_\ep-M_c),\,(n_x\cdot v)_+\right)_\p. 
\end{align*} 
By using H\"older's inequality, the trace inequality that
\begin{align*}
\|\p_tu\|_{L^2(\p\O)}\lesssim \|\p_tu\|_{H_x^1(\O)}, 
\end{align*} 
and \eqref{identity2} of Lemma~\ref{identity} with $f:=f_\ep-M_c$, we have 
\begin{align}\label{elliptic3}
\begin{split}
\left|\left( \p_tu,\,n_x\cdot\la vf_\ep^\perp\ra\right)_\p\right| 
&\lesssim \|\p_tu\|_{L^2(\p\O)} \left((1-\alpha-\beta)^2(f_\ep-M_c)^2,\,(n_x\cdot v)_+\right)_\p^\frac{1}{2}\\
&\lesssim \left(\|\p_tu\|  +\|\na_x\p_tu\|\right) \left((f_\ep-M_c)^2,\,n_x\cdot v\right)_\p^\frac{1}{2}. 
\end{split}
\end{align}
It follows from \eqref{elliptic1}, \eqref{elliptic2}, \eqref{elliptic3} and H\"older's inequality that 
\begin{align}\label{Poisson2} 
\|\p_tu\| +\|\na_x\p_tu\| \lesssim \ep^{-1}\|f_\ep^\perp\|  + \ep^{-1}\left((f_\ep-M_c)^2,\,n_x\cdot  v\right)_\p^\frac{1}{2}. 
\end{align}	
	
\medskip\noindent\textit{Step 2. Macro-micro decomposition. }\\	
Taking the inner product with $v\cdot\na_xu(t,x)$ on both sides of \eqref{rfp} yields that 
\begin{align*}
\ep\frac{\dif}{\dif t} \left( v\cdot\na_xu,\, f_\ep-M_c \right) 
=&\ \ep \left (v\cdot\na_x\p_tu,\, f_\ep-M_c \right) 
+\ep^{-1} \left( \LL_i v\cdot\na_xu,\, f_\ep-M_c \right) \\
&+\left( (v\cdot\na_x-\na_x\phi\cdot\na_v)(v\cdot\na_xu),\, f_\ep-M_c \right)  +\Sigma_\gamma, 
\end{align*}
where the boundary term 
\begin{align*}
\Sigma_\gamma:=-\left((v\cdot\na_xu)(f_\ep-M_c),\, n_x\cdot v \right)_\p. 
\end{align*} 
Taking into account the macro-micro decomposition, given as 
\begin{align*}
f_\ep-M_c=\la f_\ep\ra-M_c+f_\ep^\perp, 
\end{align*}
we obtain
\begin{align*}
\ep\frac{\dif}{\dif t} \left( v\cdot\na_xu,\, f_\ep^\perp \right) 
=&\ \ep\left( v\cdot\na_x\p_tu,\, f_\ep^\perp \right) 
-\ep^{-1}\left( v\cdot\na_xu,\, f_\ep^\perp \right) \\
&+ \left( \Delta_x u- \na_x\phi\cdot\na_xu,\, \la f_\ep\ra-M_c \right) \\
&+ \left( (v\otimes v):D^2_xu-\na_x\phi\cdot\na_xu ,\, f_\ep^\perp \right)  +\Sigma_\gamma,  
\end{align*}
where we used the facts that $\LL_iv=-v$ for $i\in\{1,2\}$ and $\la v_jv_k\ra=\delta_{jk}$ for $j,k\in\{1,\ldots,d\}$. It then turns out from the definition of $u$ and H\"older's inequality that 
\begin{align*}
\mathcal{A}+ \ep\frac{\dif}{\dif t}\left( v\cdot\na_xu,\, f_\ep^\perp \right)  
\lesssim  \left( \|D^2_xu\|  + \ep\|\na_x\p_tu\|  +\ep^{-1}\|\na_xu\| \right)\|f_\ep^\perp\| + |\Sigma_\gamma|. 
\end{align*}
By applying Lemma~\ref{identity-e} with $U:=\na_x u+c_bun_x$ and $f:=f_\ep-M_c$, as well as the trace inequality that
\begin{align*}
\|\na_xu\|_{L^2(\p\O)}\lesssim\|u\|_{H_x^2(\O)}, 
\end{align*}
we have 
\begin{align*}
|\Sigma_\gamma| 
\lesssim \left(\|u\|_{H_x^2(\O)} + \|\sqrt{c_b}u\|_{L^2(\p\O)}\right) \left((f_\ep-M_c)^2,\, n_x\cdot v \right)_\p^\frac{1}{2}, 
\end{align*}
which then implies that 
\begin{align*}
\mathcal{A}+ \ep\frac{\dif}{\dif t}\left( v\cdot\na_xu,\, f_\ep^\perp \right)  
\lesssim \,& \left( \|D^2_xu\|  + \ep\|\na_x\p_tu\|  +\ep^{-1}\|\na_xu\| \right)\|f_\ep^\perp\| \\
&+ \left(\|u\|_{H_x^2(\O)} + \|\sqrt{c_b}u\|_{L^2(\p\O)}\right) \left((f_\ep-M_c)^2,\, n_x\cdot v \right)_\p^\frac{1}{2}. 
\end{align*}
Combining this with \eqref{AA} and \eqref{Poisson2}, we arrive at 
\begin{align*}
\mathcal{A}+ \ep \frac{\dif}{\dif t}\left(v\cdot\na_xu,\, f_\ep^\perp\right) 
\lesssim&\ \|f_\ep^\perp\| \left(\ep^{-1}\mathcal{A}^\frac{1}{2} + \|f_\ep^\perp\|  +\left((f_\ep-M_c)^2,\,n_x\cdot v\right)_\p^\frac{1}{2}\right)\\
&\ +\mathcal{A}^\frac{1}{2}\left((f_\ep-M_c)^2,\, n_x\cdot v \right)_\p^\frac{1}{2}. 
\end{align*}
Applying the Cauchy-Schwarz inequality and utilizing \eqref{AA} again, we know that for some constant $c_*>0$, 
\begin{align}\label{necessario}
c_*\|\la f_\ep\ra-M_c\|^2 + \ep \frac{\dif}{\dif t} \left(v\cdot\na_xu,\, f_\ep^\perp\right)  
\lesssim \ep^{-2}\|f_\ep^\perp\|^2 + \left((f_\ep-M_c)^2,\, n_x\cdot v \right)_\p. 
\end{align}

\medskip\noindent\textit{Step 3. Construction of the modified entropy.}\\
Let $\kappa>0$ be a constant. We define the modified entropy as follows, 
\begin{align*}
E_\ep:=\|f_\ep-M_c\|^2 + \kappa \ep\left(v\cdot\na_xu,\, f_\ep^\perp\right) .
\end{align*}
If the constant $\kappa>0$ is sufficiently small, the modified entropy $E_\ep$ is equivalent to $\|f_\ep-M_c\|^2$ independent of $\ep$; indeed, by H\"older's inequality and \eqref{elliptic-regularity} of Lemma~\ref{elliptic-est}, we have 
\begin{align*}
\left|\left(v\cdot\na_xu,\, f_\ep^\perp\right) \right|
&\lesssim \|\la f_\ep\ra-M_c\| \;\! \|f_\ep^\perp\| \\
&\le \|f_\ep-M_c\|^2. 
\end{align*}
In light of this fact, we conclude from \eqref{dissipation-h} of Lemma~\ref{dissipation} (with $f_\ep$ replaced by $f_\ep-M_c$) and \eqref{necessario} that 
\begin{align*}
\frac{\dif}{\dif t} E_\ep
&\lesssim -\ep^{-2}\|f_\ep^\perp\|^2-\|\la f_\ep\ra-M_c\|^2\\
&\le -\|f_\ep-M_c\|^2\\ 
&\lesssim - E_\ep. 
\end{align*}
By using Gr\"onwall's inequality and leveraging the equivalence between $E_\ep$ and $\|f_\ep-M_c\|^2$, we arrive at the desired result. 
\end{proof}

\section{Diffusion asymptotics}\label{diffusionlimit}
This section is dedicated to proving Theorem~\ref{asympto}, which establishes the diffusion asymptotics of the kinetic model \eqref{rfp}. Throughout this section, we assume $\alpha+\beta=1$ when referring to \eqref{rfp}, specifically associated with the Maxwell boundary condition $f_\ep=\alpha\DD f_\ep+(1-\alpha)\RR f_\ep$ on $\R_+\times\Sigma_-$. 

\begin{proposition}\label{asympto0}
Let $\ep\in(0,1]$, $\fein\in L^2(\O\times\R^d,\dif m)$, and $f_\ep$ be the solution to \eqref{rfp} associated with the initial data $\fein$ and the Maxwell boundary condition. Suppose that $\rho$ is the solution to \eqref{limiteq} associated with the initial data $\rin:=\la\fein\ra$, and $\psi_\ep$ is the solution to \eqref{layereq} associated with the initial data $\psin:=\fein-\la\fein\ra$. Then we have the strong convergence that 
\begin{align}\label{asympto1}
\|f_\ep-\rho-\psi_\ep\|_{L^\infty(\R_+;\;\!L^2(\O\times\R^d,\;\!\dif m))}  \rightarrow 0 {\quad\rm as\ }\ep\rightarrow0.  
\end{align}  
If $\fein\in L_v^2H_x^2(\O\times\R^d,\dif m)$ satisfies $|v|\na_x\fein\in L^2(\O\times\R^d,\dif m)$ and the compatibility condition that $\fein$ depends only on the variable $x$ in $\p\O\times\R^d$, then we have 
\begin{align}\label{asympto2}
\|f_\ep-\rho-\psi_\ep\|_{L^\infty(\R_+;\;\!L^2(\O\times\R^d,\;\!\dif m))} 
\lesssim \sqrt{\ep} \left(\|\fein\| +\|(1+|v|)\na_x\fein\| +\|D_x^2\fein\| \right). 
\end{align}
If the initial data is well-prepared, that is, $\fein=\la\fein\ra\in H_x^2(\O)$, then we have 
\begin{align}\label{asympto3}
\|f_\ep-\rho\|_{L^\infty(\R_+;\;\!L^2(\O\times\R^d,\;\!\dif m))} 
\lesssim \sqrt{\ep} \|\fein\|_{H_x^2(\O)}. 
\end{align}
\end{proposition}

\begin{proof}
The proof will proceed in three steps. 

\medskip\noindent\textit{Step 1. Evolution of the entropy functional.}\\
We abbreviate $g_\ep:=f_\ep-\psi_\ep$. The distance between $g_\ep$ and $\rho$ reads 
\begin{align*}
\|g_\ep-\rho\|^2=\|g_\ep\|^2+\|\rho\|^2-2(g_\ep,\,\rho).
\end{align*} 
Since $\{v_1,\ldots,v_d\}$ is a part of the orthonormal basis of $L^2(\R^d,\dif\mu)$ provided by the Hermite polynomials, we have 
\begin{align*} 
|\la vf_\ep\ra|^2=|\la vf_\ep^\perp\ra|^2
\le \la |f_\ep^\perp|^2\ra.
\end{align*}
It then follows from \eqref{dissipation-h} of Lemma~\ref{dissipation} that  
\begin{align*} 
\frac{1}{2}\frac{\dif}{\dif t}\|f_\ep\|^2 
\le-\ep^{-2}\|\la vf_\ep\ra\|^2 
-(2\ep)^{-1}\left(f_\ep^2,\,n_x\cdot v\right)_\p. 
\end{align*}
Taking the evolution of $\psi_\ep$ into account, we similarly derive 
\begin{align}\label{entropy1} 
\begin{aligned}
\frac{1}{2}\frac{\dif}{\dif t}\|g_\ep\|^2
&\le-\ep^{-2}\|\la vg_\ep\ra\|^2 
-\ep^{-1}\left(g_\ep,\,v\cdot\na_x\psi_\ep\right) 
-(2\ep)^{-1}\left(g_\ep^2,\,n_x\cdot v\right)_\p\\
&=-\ep^{-2}\|\la vg_\ep\ra\|^2 
-\ep^{-1}\left(f_\ep,\,v\cdot\na_x\psi_\ep\right) 
-(2\ep)^{-1}\left(g_\ep^2-\psi_\ep^2,\,n_x\cdot v\right)_\p. 
\end{aligned}
\end{align}
By the limiting equation~\eqref{limiteq} and the macroscopic relation~\eqref{hydroFP1}, we have 
\begin{align}\label{entropy2}
\begin{aligned}
\frac{\dif}{\dif t}(g_\ep,\,\rho)
=&\, \left(\la g_\ep\ra,\,\p_t\rho\right)+\left(\p_t\la f_\ep\ra,\,\rho\right)\\
=&\, \left(\ep^{-1}\la vg_\ep\ra-\na_x\la g_\ep\ra,\,\na_x\rho\right) +\ep^{-1}\left(\rho,\,v\cdot\na_x\psi_\ep\right)\\
 &\ +\left(g_\ep,\,n_x\cdot\na_x\rho\right)_\p -\ep^{-1}\left(f_\ep\rho,\,n_x\cdot v\right)_\p, 
\end{aligned}
\end{align}
and 
\begin{align}\label{entropy3}
\begin{aligned}
\frac{1}{2}\frac{\dif}{\dif t}\|\rho\|^2 
&=\left(\rho,\,\p_t\rho\right)\\
&=-\|\na_x\rho\|^2 +\left(\rho,\,n_x\cdot\na_x\rho\right)_\p.
\end{aligned}
\end{align}
We thus obtain from \eqref{entropy1}, \eqref{entropy2}, \eqref{entropy3} that 
\begin{align*}
\frac{1}{2}\frac{\dif}{\dif t}\|g_\ep-\rho\|^2
\le& -\left\|\ep^{-1}\la vg_\ep\ra+\na_x\rho\right\|^2 \\
& + \left(\ep^{-1}\la vg_\ep\ra+\na_x\la g_\ep\ra,\,\na_x\rho\right) 
 - \left(\rho+f_\ep,\,\ep^{-1}v\cdot\na_x\psi_\ep\right) \\
& + \left( \rho-g_\ep ,\,n_x\cdot\na_x\rho\right)_\p - (2\ep)^{-1}\left(g_\ep^2-\psi_\ep^2-2f_\ep\rho,\, n_x\cdot v\right)_\p. 
\end{align*}
Substituting $g_\ep=f_\ep-\psi_\ep$ back, we obtain 
\begin{align}\label{distevolution}
\begin{aligned}
\frac{1}{2}\frac{\dif}{\dif t}\|f_\ep-\rho-\psi_\ep\|^2
\le \left(\ep^{-1}\la vf_\ep\ra+\na_x\la f_\ep\ra,\,\na_x\rho\right) +\Sigma_1 +R_\psi, 
\end{aligned}
\end{align}
where we collected the terms 
\begin{align*}
\Sigma_1:=&\, \left( \rho-f_\ep+\psi_\ep ,\,n_x\cdot\na_x\rho\right)_\p 
- (2\ep)^{-1}\left(f_\ep^2-2f_\ep\rho-2f_\ep\psi_\ep+2\rho\psi_\ep,\, n_x\cdot v\right)_\p,\\
R_\psi:=&\, -\left(f_\ep,\,\ep^{-1}v\cdot\na_x\psi_\ep\right) - \left(\na_x\psi_\ep,\,\na_x\rho\right). 
\end{align*}
Since $\la vf_\ep\ra=\la vf_\ep^\perp\ra$ and $\la (v\otimes v-{\rm I}) f_\ep\ra=\la v\otimes v f_\ep^\perp\ra$, the macroscopic relation \eqref{hydroFP2} can be recast as 
\begin{align*}
\ep^{-1}\la vf_\ep\ra+\na_x\la f_\ep\ra
= -\ep\p_t\la vf_\ep^\perp\ra -\na_x\cdot \la v\otimes vf_\ep^\perp\ra
+\na_x\phi\cdot\la v\otimes vf_\ep^\perp\ra. 
\end{align*}
Therefore, we have 
\begin{align*}
\int_0^t \left(\ep^{-1}\la vf_\ep\ra+\na_x\la f_\ep\ra,\,\na_x\rho\right)
=\ep\int_0^t\left(\la vf_\ep^\perp\ra,\,\na_x\p_t\rho\right) 
+\int_0^t\left(\la v\otimes v f_\ep^\perp\ra,\, \na_x\otimes\na_x\rho\right)&\\ 
-\ep\left(\la vf_\ep^\perp\ra,\,\na_x\rho\right)(t) +\ep\left(\la vf_\ep^\perp\ra,\,\na_x\rho\right)(0)
-\int_0^t\left(\la v\otimes v f_\ep^\perp\ra,\, n_x\otimes\na_x\rho\right)_\p&. 
\end{align*}
Combining this with \eqref{distevolution} yields that  
\begin{align}\label{entropy-relative}
\begin{aligned}
\frac{1}{2}\|f_\ep-\rho-\psi_\ep\|^2 - \int_0^t(\Sigma_2+R_\psi)
\lesssim&\left( \int_0^t \left(\ep^2\|\na_x\p_t\rho\|^2 + \|D_x^2\rho\|^2\right) \right)^{\!\frac{1}{2}}
\left(\int_0^t\|f_\ep^\perp\|^2\right)^{\!\frac{1}{2}}\\
&\,+\ep\sup\nolimits_{[0,t]}\|\na_x\rho\|\|f_\ep\| , 
\end{aligned}
\end{align}
where the new collection $\Sigma_2$ of boundary terms is defined by 
\begin{align*}
\Sigma_2:=&\ \Sigma_1+\left(({\rm I}-v\otimes v)f_\ep,\,n_x\otimes\na_x\rho\right)_\p\\
=&\,-\left((v\cdot\na_x\rho)f_\ep,\,n_x\cdot v\right)_\p
+\left(\rho+\psi_\ep,\,n_x\cdot\na_x\rho\right)_\p\\
&\,-(2\ep)^{-1}\left(f_\ep^2-2f_\ep\rho-2f_\ep\psi_\ep+2\rho\psi_\ep,\, n_x\cdot v\right)_\p. 
\end{align*}

\medskip\noindent\textit{Step 2. Taking into account boundary conditions and initial layer correction.}\\ 
We have to address the boundary term $\Sigma_2$ and the initial layer term $R_\psi$ appeared in \eqref{entropy-relative}. Let us assume the compatibility condition on the initial data that $\fein$ depends only on the variable $x$ in $\p\O\times\R^d$. By \eqref{layereq}, we know that $\psi_\ep|_{\Sigma}$ depends only on the variable $x$. Consequently, according to the boundary condition of $f_\ep$, we have  
\begin{align*}
\left(f_\ep^2-2f_\ep\rho-2f_\ep\psi_\ep+2\rho\psi_\ep,\, n_x\cdot v\right)_\p
=\left(f_\ep^2,\,n_x\cdot v\right)_\p\ge0. 
\end{align*}
It turns out that 
\begin{align*}
\Sigma_2 =-\left((v\cdot\na_x\rho)f_\ep,\,n_x\cdot v\right)_\p
-(2\ep)^{-1}\left(f_\ep^2,\,n_x\cdot v\right)_\p. 
\end{align*}
By applying Lemma~\ref{identity-e} with $U:=\na_x\rho$ and $\beta:=1-\alpha$, the trace inequality that 
\begin{align*}
\|\na_x\rho\|_{L^2(\p\O)}\lesssim\|\rho\|_{H_x^2(\O)}, 
\end{align*}
and the Cauchy-Schwarz inequality, we have 
\begin{align*}
\left|\left((v\cdot\na_x\rho)f_\ep,\,n_x\cdot v\right)_\p\right|
&\lesssim \|\rho\|_{H_x^2(\O)} \left(f_\ep^2,\, n_x\cdot v\right)_\p^\frac{1}{2}\\
&\lesssim C_1\ep \|\rho\|_{H_x^2(\O)}^2 + C_1^{-1}\ep^{-1}\left(f_\ep^2,\, n_x\cdot v\right)_\p. 
\end{align*}
Here the constant $C_1>0$ is chosen to be sufficiently large so that 
\begin{align}\label{entropy-boundary}
\Sigma_2 \lesssim \ep \|\rho\|_{H_x^2(\O)}^2.
\end{align}

In order to estimate the term $R_\psi=-\left(\ep^{-1}vf_\ep+\na_x\rho,\,\na_x\psi_\ep\right)$, we apply \eqref{layerdecay2} of Lemma~\ref{layer} with $\psi:=\na_x\psi_\ep$ and $\psin:=\na_x\fein-\na_x\la\fein\ra$ to see that for some constant $c_0>0$, 
\begin{align*}
|R_\psi| 
&\le \ep^{-1}\left(\|f_\ep\|+\|\na_x\rho\|\right)\|(1+|v|)\na_x\psi_\ep\| \\
&\lesssim \ep^{-1}e^{-c_0t/\ep^2} \left(\|f_\ep\|+\|\na_x\rho\|\right)\|(1+|v|)\na_x\fein\|. 
\end{align*}
Integrating along time and noticing the fact that $\int_0^\infty e^{-c_0t/\ep^2}\dif t=\ep^2/c_0$, we obtain 
\begin{align}\label{entropy-remainder}
\int_0^t|R_\psi| \lesssim \ep \|(1+|v|)\na_x\fein\| \sup\nolimits_{[0,t]}(\|f_\ep\|+\|\na_x\rho\|). 
\end{align} 

\medskip\noindent\textit{Step 3. Conclusion and approximation.}\\ 
We are in a position to leverage the dissipation and the regularity of the limiting equation to draw our conclusion. Recall the parabolic estimate that for any $\rho$ solving \eqref{limiteq} and any $t>0$, 
\begin{align*}
\sup\nolimits_{[0,t]}\left(\|\rho\|_{H_x^1(\O)}^2+\|\p_t\rho\|^2\right) + \int_0^t\left(\|D_x^2\rho\|^2 +\|\na_x\p_t\rho\|^2\right) \lesssim \|\rin\|_{H^2(\O)}^2; 
\end{align*}
one may refer to \cite[\S~7.1.3]{Evans} in this regard. 
It then follows from \eqref{entropy-relative}, \eqref{entropy-boundary}, \eqref{entropy-remainder} that 
\begin{align*}
\|f_\ep-\rho-\psi_\ep\|^2
\lesssim \|\rin\|_{H^2(\O)} \left(\int_0^t\|f_\ep^\perp\|^2\right)^{\!\frac{1}{2}} +\ep\|\rin\|_{H^2(\O)}\sup\nolimits_{[0,t]}\|f_\ep\| +\ep\|\rin\|_{H^2(\O)}^2\\
+\ep\|(1+|v|)\na_x\fein\| \sup\nolimits_{[0,t]}\left(\|f_\ep\|+\|\rin\|_{H^2(\O)}\right)&.  
\end{align*} 
By the dissipation estimate \eqref{dissipationR} of Lemma~\ref{dissipation} and the Cauchy-Schwarz inequality, we have  
\begin{align}\label{fepr}
\|f_\ep-\rho-\psi_\ep\|^2
\lesssim \ep \left(\|\fein\|^2 +\|(1+|v|)\na_x\fein\|^2 +\|D_x^2\fein\|^2 \right).  
\end{align} 
We thus obtain \eqref{asympto2} under the compatibility assumption on $\fein$. We remark that \eqref{asympto3} can be derived in a similar manner, based on the observation that the well-preparedness of $\fein$ implies $\psi_\ep\equiv0$. 

As for general initial data $\fein\in L^2(\O\times\R^d,\dif m)$, we consider its approximation 
\begin{align*}
\fein^N:=\varphi_N\fein*\varrho_N. 
\end{align*}
Here $N\in\mathbb{N}_+$, $\varrho_N(x):=N^d\varrho_1(Nx)$ with $0\le\varrho_1(x)\in C^\infty_0(\R^d)$ such that $\int_{\R^d}\varrho_1(x)\dif x=1$, and $\varphi_N(x,v)\in C^\infty_0(\O\times\R^d)$ satisfies 
\begin{align*}
0\le\varphi_N(x,v)\le1,\quad
|\na_x\varphi_N(x,v)|\lesssim N,\quad
|D^2_x\varphi_N(x,v)|\lesssim N^2 {\quad\rm in\ } \O\times\R^d, \\
\varphi_N(x,v)=1 {\quad\rm in\ }
\{(x,v)\in\O\times\R^d:\, {\rm dist}(x,\p\O)\ge N^{-1} {\rm\ and\ } |v|\le N\},\\
\varphi_N(x,v)=0 
{\quad\rm in\ }\{(x,v)\in\O\times\R^d:\, {\rm dist}(x,\p\O)\le (2N)^{-1} {\rm\ or\ } |v|\ge 2N\}. 
\end{align*}  
In this setting, we know that 
\begin{align*}
\|\fein-\fein^N\|\rightarrow0 {\quad\rm as\ } N\rightarrow\infty,  
\end{align*}
and 
\begin{align*}
\|(1+|v|)\na_x\fein^N\| + \|D_x^2\fein^N\|
\lesssim \|(|v||\na_x\varphi_N|+|D_x^2\varphi_N|)\fein*\varrho_N\| \\
+ \|(|v|\varphi_N+|\na_x\varphi_N|)\fein*(\na_x\varrho_N)\| 
+ \|\varphi_N\fein*(D_x^2\varrho_N)\|&.
\end{align*} 
By using Young's inequality, we obtain
\begin{align}\label{fepr1}
\begin{aligned}
\|(1+|v|)\na_x\fein^N\| + \|D_x^2\fein^N\|\\
\lesssim\left(N^2\|\varrho_N\|_{L^1(\R^d)} +N\|\na_x\varrho_N\|_{L^1(\R^d)} +\|D_x^2\varrho_N\|_{L^1(\R^d)}\right) \|\fein\|\\
\lesssim N^2\|\fein\|&. 
\end{aligned}
\end{align} 
Let $f_\ep^N$, $\rho^N$, $\psi_\ep^N$ be the solutions to \eqref{rfp}, \eqref{limiteq}, \eqref{layereq} associated with the initial conditions
\begin{align*}
f_\ep^N|_{t=0}&=\fein^N,\\
\rho^N|_{t=0}&=\la\fein^N\ra,\\
\psi_\ep^N|_{t=0}&=\fein^N-\la\fein^N\ra. 
\end{align*} 
In view of the dissipation estimates for \eqref{rfp} and \eqref{layerdecay1} provided by Lemma~\ref{dissipation} and Lemma~\ref{layer} respectively, as well as the standard dissipation property for \eqref{limiteq}, we have 
\begin{align}\label{fepr2}
\begin{aligned}
\|f_\ep-f_\ep^N\| \le \|\fein-\fein^N\|,\\
\|\rho-\rho^N\| \le \|\la\fein\ra-\la\fein^N\ra\|\le \|\fein-\fein^N\|, \\
\|\psi_\ep-\psi_\ep^N\| \le \|\fein-\la\fein\ra-\fein^N+\la\fein^N\ra\|
\le 2\|\fein-\fein^N\|. 
\end{aligned}
\end{align} 
According to \eqref{fepr}, \eqref{fepr1}, \eqref{fepr2}, we obtain  
\begin{align*}
\|f_\ep-\rho-\psi_\ep\| 
&\le \|f_\ep-f_\ep^N\| + \|\rho-\rho^N\|+ \|\psi_\ep-\psi_\ep^N\| + \|f_\ep^N-\rho^N-\psi_\ep^N\| \\
&\lesssim \|\fein-\fein^N\| + \sqrt{\ep} N^2\|\fein\|.  
\end{align*} 
Sending $\ep\rightarrow0$ and then $N\rightarrow\infty$ yields \eqref{asympto1}. This completes the proof. 
\end{proof}

Theorem~\ref{asympto} is a direct consequence of Proposition~\ref{asympto0}. 

\begin{proof}[Proof of Theorem~\ref{asympto}]
By the dissipation estimate \eqref{layerdecay1} of Lemma~\ref{layer} and Gr\"onwall's inequality, we know that 
\begin{align*}
\|\psi_\ep\| \le e^{-t/\ep^2}\|\psin\|
\le e^{-t/\ep^2}\|\fein\|. 
\end{align*}
On the one hand, integrating along time, we obtain  
\begin{align*}
\|\psi_\ep\|_{L^2(\R_+\times\O\times\R^d,\;\!\dif t\dif m)} 
\le \|e^{-t/\ep^{2}}\|_{L^2(\R_+)} \|\fein\|
\le \ep\|\fein\|, 
\end{align*}
so that for any $T>0$, we have 
\begin{align}\label{psi1}
\|f_\ep-\rho\|_{L^2([0,T]\times\O\times\R^d,\;\!\dif t\dif m)} 
\le \sqrt{T}\|f_\ep-\rho-\psi_\ep\|_{L^\infty(\R_+;\;\!L^2(\O\times\R^d,\;\!\dif m))} + \ep\|\fein\|.  
\end{align}
On the other hand, for any $t\ge\tau(\ep)$, we have  
\begin{align}\label{psi2}
\begin{aligned}
\|(f_\ep-\rho)(t)\| 
&\le \|(f_\ep-\rho-\psi_\ep)(t)\| + \|\psi_\ep(t)\|\\
&\le \|(f_\ep-\rho-\psi_\ep)(t)\| + e^{-\tau(\ep)/\ep^2}\|\fein\|.
\end{aligned} 
\end{align}
In view of \eqref{psi1} and \eqref{psi2}, together with Proposition~\ref{asympto0}, we conclude the proof. 
\end{proof}

\appendix
\section{Elliptic estimates under the Robin boundary condition}\label{append}
This appendix is dedicated to clarifying the elliptic estimates used in Section~\ref{hypocoercive}. Though standard, we include their proofs below for the sake of completeness. 

\begin{lemma}\label{elliptic-est}
Let $\O$ be a bounded $C^{1,1}$-domain, the function $\phi:\O\rightarrow\R$ be Lipschitz, the source term $S\in L^2(\O)$, and the coefficient $c_b:\p\O\rightarrow[0,\infty)$. Then there exists a unique solution $u$ to 
\begin{align}\label{elliptic-S}
\left\{ 
\begin{aligned}
\ & u-\Delta_x u+ \na_x\phi\cdot\na_xu=S {\quad\rm in\ }\O, \\
\ & n_x\cdot\na_xu +c_bu =0 {\quad\rm on\ }\p\O. \\
\end{aligned}
\right. 
\end{align} 
Supposing that $c_b$ is Lipschitz and satisfies $|\na c_b|\le C_0 c_b$ on $\p\O$ for some constant $C_0>0$, we then have the regularity estimate 
\begin{align}\label{elliptic-regularity}
\|u\| + \|\na_x u\| + \|D^2_x u\| + \|\sqrt{c_b}u\|_\p \le C \|S\| .  
\end{align}
If additionally either $c_b|_{\p\O}=\left(S,\,1\right)=0$, or $c_b|_{\p\O}\ge\delta_b$ for some constant $\delta_b>0$, then we have the Poincar\'e type inequality
\begin{align}\label{elliptic-Poincare}
\|u\|  \le C\left( \|\na_x u\|  + \|\sqrt{c_b}u\|_\p \right). 
\end{align}
Furthermore, under the above assumptions, we have 
\begin{align}\label{elliptic-z}
\|S\|  + \|u\|  + \|\na_xu\|  + \|D_x^2u\|  + \|\sqrt{c_b}u\|_\p 
\le C \left(S-u,\,S\right)^\frac{1}{2} . 
\end{align} 
Here the constant $C>0$ depends only on $d,\O,\|\phi\|_{C^{0,1}(\O)},C_0,\delta_b$. 
\end{lemma}

\begin{remark}
For any Lipschitz function $c_b:\p\O\rightarrow[0,\infty)$ such that $|\na c_b|\le C_0c_b$ on $\p\O$ and $c_b(x_0)=0$ for some constant $C_0>0$ and some point $x_0\in\p\O$, it can be shown that $c_b\equiv0$ on $\p\O$. In light of this observation, we may assume that $c_b$ is strictly positive whenever it is not identically zero.
\end{remark}

\begin{proof}
First of all, one may find the existence and uniqueness results in \cite[Corollary 5.27]{Lieberman}. One may also refer to \cite[\S~5.7]{Lieberman} for the regularity issues. We sketch the proof of \eqref{elliptic-regularity} in our scenario. Indeed, taking the inner product with $u$ on both sides of \eqref{elliptic-S} yields that 
\begin{align*}
\|u\|^2 + \|\na_xu\|^2 
&= \left(S,\,u\right)  + \left(n_x\cdot\na_xu,\,u\right)_\p\\
&= \left(S,\,u\right) - \|\sqrt{c_b}u\|_\p^2, 
\end{align*}
which implies from the Cauchy-Schwarz inequality that 
\begin{align}\label{eu1}
\|u\| + \|\na_xu\| + \|\sqrt{c_b}u\|_\p \lesssim \|S\| . 
\end{align}
In order to derive higher regularity of $u$, we consider $R>0$ and $\overline{x}\in\p\O$, and assume for simplicity that $\p\O\cap B_R(\overline{x})=\{x\in B_R(\overline{x}):x_d=0\}$ and $\O\cap B_R(\overline{x})=\{x\in B_R(\overline{x}):x_d>0\}$; it can be verified that a similar computation as the one below remains applicable for general curved $C^{1,1}$-domains based on a boundary flattening argument. Pick a cut-off function $\eta\in C^\infty_0(B_R(\overline{x}))$ valued in $[0,1]$ such that $\eta|_{B_{R/2}(\overline{x})}=1$. Let us abbreviate $\p_k=\p_{x_k}$ for $k\in\{1,\ldots,d\}$. Taking the inner product with $\eta^2\p_ku$ on both sides of the equation satisfies by $\p_ku$ yields that for any $k\in\{1,\ldots,d-1\}$, 
\begin{align*}
\left(\eta^2\p_ku,\,\p_ku\right) + \left(\na_x(\eta^2\p_ku),\,\na_x\p_ku\right) 
=&\left(\p_k(\eta^2\p_ku),\,\na_x\phi\cdot\na_xu\right)
+\left(\eta^2\p_ku,\,\na_x\phi\cdot\na_x\p_ku\right) \\
&-\left(\p_k(\eta^2\p_ku),\,S\right) -\left(\eta^2\p_ku,\,\p_k(c_bu)\right)_\p. 
\end{align*}
Here we notice that the boundary term 
\begin{align*}
\left(\eta^2\p_ku,\,\p_k(c_bu)\right)_\p =\left(\eta^2\p_ku,\, u\p_kc_b \right)_\p + \|\sqrt{c_b}\eta\p_ku\|_\p^2,
\end{align*} 
and due to the assumption that $|\p_kc_b|\le C_0|c_b|$, 
\begin{align*}
\left|\left(\eta^2\p_ku,\, u\p_kc_b \right)_\p\right| \le C_0\|\sqrt{c_b}\eta u\|_\p \|\sqrt{c_b}\eta\p_ku\|_\p.
\end{align*}
Applying the Cauchy-Schwarz inequality, we deduce 
\begin{align*}
\|\eta\na_x\p_ku\| + \|\sqrt{c_b}\eta\p_ku\|_\p
\lesssim \|(\eta+|\na_x\eta|)\p_ku\| + \|\eta S\| +\|\sqrt{c_b}\eta u\|_\p. 
\end{align*}
We observe from \eqref{elliptic-S} that 
\begin{align*}
\p_d^2u=-S+u-\sum\nolimits_{k=1}^{d-1}\p_k^2u+\na_x\phi\cdot\na_xu
\end{align*}
so that 
\begin{align}\label{eu2}
\|D_x^2u\| \lesssim \|\na_xu\| + \|S\| +\|\sqrt{c_b}u\|_\p. 
\end{align}
Gathering \eqref{eu1} and \eqref{eu2}, we arrive at \eqref{elliptic-regularity}. 

As for \eqref{elliptic-Poincare}, we first assume that $c_b|_{\p\O}=\left(S,\,1\right)=0$. Combining the Poincar\'e inequality that
\begin{align*}
\|u\|\lesssim\|\na u\|+|(u,\,1)|
\end{align*}
with the fact that 
\begin{align*}
(u,\,1) = (n_x\cdot\na_xu,\,1)_\p =0, 
\end{align*}
we obtain \eqref{elliptic-Poincare}. Otherwise, for $c_b|_{\p\O}\ge\delta_b>0$, the Poincar\'e inequality (see for instance \cite[Lemma~5.22]{Lieberman}) that 
\begin{align*}
\|u\| \lesssim \|\na_xu\| + \|u\|_\p
\end{align*}
also implies \eqref{elliptic-Poincare}. 

Finally, in view of \eqref{elliptic-S}, we notice that 
\begin{align*}
\left(S-u,\,S\right)  &= \left(-\Delta_x u+ \na_x\phi\cdot\na_xu,\,u-\Delta_x u+ \na_x\phi\cdot\na_xu\right) \\
&\ge \left(-\Delta_x u+ \na_x\phi\cdot\na_xu,\,u\right)\\
&= \|\na_xu\|^2 - \left(n_x\cdot\na_xu ,\,u\right)_\p\\
&= \|\na_xu\|^2 + \|\sqrt{c_b}u\|_\p^2, 
\end{align*} 
and
\begin{align*}
\left(S-u,\,S\right)  &= \|S\|^2 - \left(u,\,u-\Delta_x u+ \na_x\phi\cdot\na_xu\right) \\
&=\|S\|^2 - \|u\|^2 - \|\na_xu\|^2 - \|\sqrt{c_b}u\|_\p^2.  
\end{align*} 
Combining these two facts with \eqref{elliptic-regularity} and \eqref{elliptic-Poincare}, we conclude \eqref{elliptic-z}. 
The proof is complete. 
\end{proof}

%\bibliographystyle{alpha}
%\bibliography{D23}
%\end{document}

\newcommand{\etalchar}[1]{$^{#1}$}

\end{document}